\documentclass[reqno]{amsart}
\usepackage{graphicx,amssymb,hyperref,color}

\DeclareMathOperator{\pdet}{pdet}
\DeclareMathOperator{\rk}{rk}

\newtheorem{theorem}{Theorem}
\numberwithin{theorem}{section}
\newtheorem{proposition}[theorem]{Proposition}

\newtheorem{corollary}[theorem]{Corollary}
\newtheorem{lemma}[theorem]{Lemma}

\theoremstyle{definition}
\newtheorem{definition}[theorem]{Definition}
\newtheorem{remark}[theorem]{Remark}
\newtheorem{example}[theorem]{Example}

\newcommand{\PHX}{{\phantom{X}}} 
\newcommand{\multiset}[1]{\{\!\{#1\}\!\}}
\newcommand{\bigmultiset}[1]{\left\{\!\left\{#1\right\}\!\right\}}

\newcommand{\abs}[1]{\lvert#1\rvert}

\newcommand{\HH}{\tilde{H}}

\newcommand{\Qq}{\mathbb{Q}}
\newcommand{\Rr}{\mathbb{R}}
\newcommand{\Zz}{\mathbb{Z}}

\newcommand{\xx}{\mathbf{x}}

\newcommand{\betti}{\tilde\beta}
\newcommand{\bd}{\partial}
\newcommand{\cbd}{\partial^*}
\newcommand{\wbd}{\hat\partial}
\newcommand{\wcbd}{\hat\partial^*}

\newcommand{\x}{\times}
\newcommand{\st}{\colon}
\newcommand{\sm}{\setminus}

\newcommand{\0}{\emptyset}
\newcommand{\disun}{\mathbin{\dot{\cup}}}

\newcommand{\SST}{{\mathcal T}}
\newcommand{\sud}{{\mathbf s}^{\rm ud}} 
\newcommand{\sdu}{{\mathbf s}^{\rm ud}} 
\newcommand{\stot}{{\mathbf s}^{\rm tot}} 
\newcommand{\Lud}{L^{\rm ud}} 
\newcommand{\Ldu}{L^{\rm du}} 
\newcommand{\Swtot}{\hat{\mathbf s}^{\rm tot}} 
\newcommand{\Swud}{\hat{\mathbf s}^{\rm ud}} 
\newcommand{\Swdu}{\hat{\mathbf s}^{\rm du}} 
\newcommand{\Lwtot}{\hat{L}^{\rm tot}} 
\newcommand{\Lwud}{\hat{L}^{\rm ud}} 
\newcommand{\Lwdu}{\hat{L}^{\rm du}} 

\begin{document}
\title[A weighted cellular matrix-tree theorem]{A weighted cellular matrix-tree theorem, with applications to complete colorful and cubical complexes}

\author[G. Aalipour]{Ghodratollah Aalipour}
\address{School of Mathematical Sciences, Rochester Institute of Technology}
\email{ga5481@rit.edu}
\author[A. M. Duval]{Art M.\ Duval}
\address{Department of Mathematical Sciences, University of Texas at El Paso}
\email{aduval@utep.edu}
\author[W. Kook]{Woong Kook}
\address{Department of Mathematical Sciences, Seoul National University, Seoul, Republic of Korea}
\email{woongkook@snu.ac.kr}
\author[K.-J. Lee]{Kang-Ju Lee}
\address{Department of Mathematical Sciences, Seoul National University, Seoul, Republic of Korea}
\email{leekj0706@snu.ac.kr}
\author[J. L. Martin]{Jeremy L.\ Martin}
\address{Department of Mathematics, University of Kansas}
\email{jlmartin@ku.edu}
\date{\today}
\thanks{GA thanks the Department of Mathematical Sciences at the University of Texas at El~Paso for support as a Visiting Research Scholar in 2012--2013, when this work started.  WK and KL are partially supported by Basic Science Research Program through the National Research Foundation of Korea (NRF) funded by the Ministry of Education (0450-20160023), by Samsung Science and Technology Foundation under Project Number SSTF-BA1402-08.  JLM acknowledges support from a Simons Foundation Collaboration Grant.
}

\subjclass[2010]{%
05C05, 
05C50, 
05E45} 
\begin{abstract}
We present a version of the weighted cellular matrix-tree theorem that is suitable for calculating explicit generating functions for spanning trees of highly structured families of simplicial and cell complexes.  We apply the result to give weighted generalizations of the tree enumeration formulas of Adin for complete colorful complexes, and of Duval, Klivans and Martin for skeleta of hypercubes. We investigate the latter further via a logarithmic generating function for weighted tree enumeration, and derive another tree-counting formula using the unsigned Euler characteristics of skeleta of a hypercube.
\end{abstract}

\keywords{matrix-tree theorem, Laplacian, complete colorful complex, hypercube, Euler characteristic
} 
\maketitle

\section{Introduction}

The matrix-tree theorem, first discovered by Kirchhoff in 1845, expresses the number of spanning trees of a (finite, undirected) graph in terms of the spectrum of its Laplacian matrix.  It can be used to derive closed formulas for the spanning tree counts of numerous families of graphs such as complete, complete bipartite, complete multipartite, hypercube and threshold graphs; see, e.g., \cite[\S5.6]{EC2} and \cite[Chapter 5]{Moon}.  The matrix-tree theorem has a natural weighted analogue that expresses the generating function for spanning trees in terms of the spectrum of a weighted Laplacian matrix.  For certain graphs with tight internal structure, the associated tree generating functions for statistics such as degree sequence have explicit factorizations which can be found by examination of the weighted Laplacian spectra.

Central to the matrix-tree theorem is the characterization of a spanning tree of a graph as a set of edges corresponding to a column basis of its incidence matrix.  This observation holds true in the more general context of finite simplicial and CW complexes, an idea introduced by Bolker \cite{Bolker} and Kalai \cite{Kalai} and recently studied by many authors; see~\cite{Chapter} for a survey.  The matrix-tree theorem and its weighted versions extend to this broader context, raising the question of finding explicit formulas for generating functions for spanning trees in highly structured CW complexes.  Specifically, for a CW complex $\Delta$ of dimension $\geq k$, let $\{X_\sigma\}$ be a set of commuting indeterminates corresponding to the cells $\sigma\in\Delta$, let $\SST_k(\Delta)$ denote its set of $k$-dimensional spanning trees, and let $\tilde H$ denote reduced cellular homology.  The higher-dimensional analogue of the (unweighted) tree count is
\[\tau_k(\Delta)=\sum_{\Upsilon\in\SST_k(\Delta)}\abs{\HH_{k-1}(\Upsilon;\Zz)}^2\]
and the corresponding generating function (the \emph{weighted tree count}) is
\[\hat{\tau}_k(\Delta) = \sum_{\Upsilon\in\SST_k(\Delta)} \abs{\HH_{k-1}(\Upsilon;\Zz)}^2 \prod_{\sigma\in\Upsilon_k} X_\sigma;\]
the homology-squared summands in each case arise from the proof of the matrix-tree theorem \cite{Kalai,DKM1}, and each summand simply equals 1 when $k=1$.

The indeterminates $X_\sigma$ may be further specialized.  Kalai \cite{Kalai} calculated $\tau_k$ and $\hat{\tau}_k$ for \emph{skeleta of simplices} (see~\eqref{kalai-unweighted} and~\eqref{Kalai-weight} below), respectively generalizing Cayley's formula $n^{n-2}$ for the graph case ($k=1$) and the degree-sequence generating function that can be obtained from the well-known Pr\"ufer code.  Subsequently, Duval, Klivans and Martin \cite{DKM1} gave weighted tree counts for \emph{shifted simplicial complexes} (which generalize threshold graphs), by refining Duval and Reiner's formula for their Laplacian eigenvalues \cite{Duval-Reiner}.  Unweighted tree counts for the families of \emph{complete colorful complexes} (which generalize complete bipartite and multipartite graphs) were found by Adin~\cite{Adin} (see equation~\eqref{adin-unweighted} below) and and for \emph{hypercubes} by Duval, Klivans and Martin~\cite{DKM2}.  One goal of this article is to calculate weighted tree counts for these complexes.

Our main technical tool is the following weighted version of the matrix-tree theorem.  Let $\Delta$ be a $d$-dimensional cell complex such that $\HH_i(\Delta;\Qq)=0$ for all $i<d$.  Let $\bd_k$ be the cellular boundary map from $k$-chains to $(k-1)$-chains (or the matrix representing it)  and let $\cbd_k$ be the corresponding coboundary map.  Also, let $D_i$ be the diagonal matrix with entries $X_\sigma^{1/2}$, as $\sigma$ ranges over the $i$-dimensional cells of $\Delta$.  The \emph{weighted boundary} and \emph{weighted coboundary} matrices are
\begin{align*}
\wbd_k&=D_{k-1}^{-1}\bd_kD_k, & \wcbd_k&=D_k\cbd_kD_{k-1}^{-1},
\end{align*}
and the (up-down) \emph{weighted Laplacian} is $\Lwud_{k-1}=\wbd_k\wcbd_k$ (so indexed because it acts on the space of $(k-1)$-chains).
Finally, define $\hat{\pi}_k(\Delta)$ to be the product of nonzero eigenvalues of~$\Lwud_{k-1}$.

\begin{theorem} \label{thm:CMTT}
With the foregoing assumptions, for all $0 \leq k \leq d$ we have
\[\hat{\pi}_k(\Delta)=\frac{\hat{\tau}_k(\Delta)\; \hat{\tau}_{k-1}(\Delta)}{\abs{\HH_{k-2}(\Delta)}^2\; X_{(k-1)}}\]
for every $k$, where
\[X_{(k-1)}=\prod_{\sigma\in\Delta_{k-1}} X_\sigma.\]
\end{theorem}

This formula is a slight generalization of~\cite[Theorem~5.3]{MMRW} (in which it was assumed that $\HH_i(\Delta;\Zz)=0$ for all $i<d$)
and was obtained independently by the first two authors.  An analogous weighted formula, in which the weighted boundary maps were defined as $\bd_k D_k$, was stated in \cite[Theorem~2.12]{DKM2}, but the present version has two key advantages.  First, it is easy to check that $\wbd_k \wbd_{k+1}=0$ (i.e., the maps $\wbd_k$ form a chain complex), which leads directly to a useful identity on spectra of weighted Laplacians (see Section~\ref{section:ccc} below).  Second, Theorem~\ref{thm:CMTT} lends itself well to inductive calculations of weighted tree counts for an entire family of complexes.  (The formula of \cite{DKM2} expresses $\hat{\pi}_k$ in terms of a mix of weighted and unweighted tree counts, which is less convenient to work with.)  Comparable weighted cellular matrix-tree theorems (in which the weights carry physical interpretations) appear in \cite[Theorem~C]{CCK1}, \cite[Theorem~C]{CCK2}.  We prove Theorem~\ref{thm:CMTT} in Section~\ref{section:proof}.

Section~\ref{section:ccc} is concerned with \emph{complete colorful complexes}, which we describe briefly.  For positive integers $n_1,\dots,n_r$ and a family of pairwise disjoint vertex sets (``color classes'') $V_1,\dots,V_r$ with $\abs{V_i}=n_i$, the complete colorful complex $\Delta=\Delta_{n_1,\dots,n_r}$ is the simplicial complex on $V=V_1\cup\cdots\cup V_r$ consisting of all faces with no more than one vertex of each color.  For example, $\Delta_{n_1,n_2}$ is just a complete bipartite graph.  Complete colorful complexes were studied by Adin~\cite{Adin}, who calculated the numbers $\tau_k(\Delta)$ for all $k\leq r-1$.  Here we apply Theorem~\ref{thm:CMTT} to prove a weighted version of Adin's formula.  Set $X_F=\prod_{v\in F} X_v$ for every face $F$, so that for a pure subcomplex $\Upsilon\subseteq\Delta$ of dimension~$k$, we have $\prod_{F\in\Upsilon_k} X_F=\prod_v X_v^{\deg_\Upsilon(v)}$, where $\deg_\Upsilon(v)$ is the number of $k$-faces of~$\Upsilon$ that contain $v$.  Then $\hat{\tau}_k(\Delta)$ becomes the \emph{degree-weighted tree count}, that is, the generating function for spanning trees of  $\Delta$ by their vertex-facet degree sequence.

\begin{theorem} \label{complete-colorful-theorem}
The degree-weighted tree count for $\Delta=\Delta_{n_1,\dots,n_r}$ is
\begin{multline*}
\hat{\tau}_k(\Delta) := \sum_{\Upsilon\in\SST_k(\Delta)} \abs{\HH_{k-1}(\Upsilon;\Zz)}^2 \prod_{v\in V} X_v^{\deg_\Upsilon(v)} \\
 = \left(\prod_{q=1}^r P_q^{E_{k,q}}\right)\prod_{\substack{J\subseteq [r]\\ \abs{J}\leq k}}
  \left(\sum_{q\in[r]\sm J} S_q \right)^{\binom{r-2-\abs{J}}{k-\abs{J}} \prod_{t \in J}(n_t - 1)}
\end{multline*}
where
\begin{align*}
S_q&=\sum_{v\in V_q}X_v, &
P_q&=\prod_{v\in V_q}X_v, &
E_{k,q}&=\sum_{\substack{J\subseteq [r]\sm\{q\}:\\ \abs{J} \leq k-1}}(-1)^{k-1-\abs{J}}\prod_{t\in J}n_t.
\end{align*}
\end{theorem}
This formula generalizes Adin's unweighted count (which can be obtained by setting $X_v=1$ for all $v\in V$), as well as known  formulas for weighted and unweighted spanning tree counts for complete multipartite graphs (the case $k=1$).

Sections~\ref{section:hypercubes} and~\ref{sec-twoinvariants} are concerned with the \emph{hypercube} $Q_n=[0,1]^n\subset\Rr^n$, which has a natural CW structure with $3^n$ cells of the form $\sigma=(\sigma_1,\dots,\sigma_n)$ with $\sigma_i\in\{[0,1],0,1\}$. Let $q_1,\dots,q_n,y_1,\dots,y_n,z_1,\dots,z_n$ be commuting indeterminates, and assign to each face
$\sigma\in Q_n$ the monomial weight
\[X_\sigma=
\left(\prod_{i:\ \sigma_i=[0,1]} q_i\right)
\left(\prod_{i:\ \sigma_i=0} y_i\right)
\left(\prod_{i:\ \sigma_i=1} z_i\right).\]
In Section~\ref{section:hypercubes}, we apply Theorem~\ref{thm:CMTT} to
prove the following weighted enumeration formula for spanning trees of hypercubes, which appeared as~\cite[Conjecture~4.3]{DKM2}.

\begin{theorem} \label{thm:cube-tree}
Let $n\geq k\geq 0$ be integers.  Then the $k$-th weighted tree count of $Q_n$ is
\begin{multline*}
\hat{\tau}_k(Q_n) := \sum_{\Upsilon\in\SST_k(Q_n)} \abs{\HH_{k-1}(\Upsilon;\Zz)}^2 \prod_{\sigma\in\Upsilon_k} X_\sigma \\
= (q_1\cdots q_n)^{\sum_{i=k-1}^{n-1}\binom{n-1}{i}\binom{i-1}{k-2}}
\prod_{\substack{S\subseteq[n]\\ \abs{S}>k}}
\left(\sum_{i\in S} q_i(y_i+z_i)\prod_{j\in S\sm i}y_jz_j\right)^{\binom{\abs{S}-2}{k-1}}.
\end{multline*}
\end{theorem}
Setting $k=1$ recovers the weighted tree count for hypercube graphs found by Martin and Reiner~\cite[Theorem~3]{JLM-VR-Facto}; see also Remmel and Williamson~\cite{RW}.

In general, any family of cell complexes with a strong recursive structure ought to be a good candidate for application of Theorem~\ref{thm:CMTT}. Examples include  the shifted cubical complexes described in \cite[\S5]{DKM2} and joins of $k$-skeleta of simplices (which generalize complete colorful complexes, the case $k=0$).

Section~\ref{sec-twoinvariants} is concerned with a \emph{logarithmic} approach to weighted enumeration of trees in a hypercube $Q_n$, focusing on the multiplicities of the weights. Logarithmic generating functions for unweighted tree counts of an acyclic cell complex were given in~\cite[Theorem~8]{KK}, and for weighted tree counts of a $\Zz$-APC simplicial complex (defined in Section~\ref{sec-prelim}) in~\cite[Theorem~5]{KL15}.  This gives a different formula for $\hat{\tau}_k(Q_n)$, in terms of the reduced Euler characteristics of skeleta of $Q_{n-1}$. This formula can be reduced to Theorem~\ref{thm:cube-tree} together with several identities. Let $\abs{\tilde{\chi}(K)}$ denote the unsigned reduced Euler characteristic for a cell complex $K$. Let $q_{[n]}=q_1\cdots q_n$, and define $y_{[n]}$ and $z_{[n]}$ similarly.

\begin{theorem}\label{inter-hyper-theorem} For $k \in [n]$, the $k$-th weighted tree count of $Q_n$ is
\begin{equation*}\label{eq:cube-euler}
\hat{\tau}_k(Q_n) = \big(q_{[n]}\big)^{\abs{\tilde{\chi}({Q_{n-1}^{(k-2)}})}}\big(y_{[n]}z_{[n]}\big)^{\abs{\tilde{\chi}({Q_{n-1}^{(k-1)}})}} \prod_{\substack{S\subseteq[n]\\ \abs{S}>k}}{u_S}^{{\binom{\abs{S}-2}{k-1}}}
\end{equation*}
where \(u_S=\sum_{i \in S} q_i\left(y_i^{-1}+z_i^{-1}\right).\)
Moreover, $\hat{\tau}_0(Q_n)=\prod_{i \in [n]}{(y_i+z_i)}$.
\end{theorem}

The authors thank Vic Reiner for helpful conversations, and an anonymous referee for a careful reading of the manuscript.

\section{Preliminaries}\label{sec-prelim}

For an integer $n$, the symbol $[n]$ denotes the set $\{1,\dots,n\}$.  The notation $\multiset{x_1\colon m_1;\ \dots;\ x_n\colon m_n}$ denotes the multiset in which each element~$x_i$ appears with multiplicity~$m_i$.  If a multiplicity is omitted, it is assumed to be~1.
The union of two multisets is defined by adding multiplicities, element by element.

We adopt the following convention for binomial coefficients with negative arguments:
\begin{equation} \label{binomial-coeffs}
\binom{n}{k}=\begin{cases}
\frac{n(n-1)\cdots(n-k+1)}{k!} & \text{ if } k>0,\\
1 & \text{ if } n=k \text{ or } k=0,\\
\binom{n+1}{k+1}-\binom{n}{k+1} & \text{otherwise.}
\end{cases}
\end{equation}
This is the standard definition for $k\geq 0$, and it satisfies the Pascal recurrence $\binom{n}{k}=\binom{n-1}{k}+\binom{n-1}{k-1}$ for all $n,k\in\Zz$, with the single exception $n=k=0$.  This convention will be useful in the proofs of  
Theorem \ref{complete-colorful-theorem}, Theorem \ref{thm:cube-tree} and Lemma \ref{standardized-identity}. 
In particular, $\binom{n}{k}=0$ whenever $k<0\leq n$,  $0\leq n<k$, or $n<k<0$. 
Moreover, note that $\binom{n}{0}=1$ and  $\binom{n}{1}=n$ for all $n\in\Zz$.

We will state our results in the general setting of cell (CW) complexes as far as possible.  A reader more comfortable with simplicial complexes may safely replace ``cell complex'' by ``simplicial complex'' throughout, as most of the basic topological facts about simplicial complexes~\cite[pp.~19--24]{Stanley-CCA} have natural analogues in the cellular setting; see, e.g., \cite[\S2.2]{Hatcher}.

Throughout the paper, $\Delta$ will be a cell complex of dimension~$d$ with finitely many cells.  We use the standard symbols $C_k(\Delta;R)$, $\bd_k(\Delta)$, $f_k(\Delta)$, $\HH_k(\Delta;R)$ and $\betti_k(\Delta)$ for the chain groups, cellular boundary maps, face numbers, reduced homology and reduced Betti numbers of a cell complex $\Delta$, dropping $\Delta$ from the notation when convenient.  If the coefficient ring $R$ is omitted, it is assumed to be $\Zz$.  The set of $k$-dimensional cells is denoted by $\Delta_k$, and the $k$-skeleton by $\Delta^{(k)}$.  We say that $\Delta$ is \emph{acyclic in positive codimension} (APC) if $\HH_k(\Delta;\Qq)=0$ (equivalently, $\betti_k(\Delta)=0$ or $\abs{\HH_k(\Delta;\Zz)}<\infty$) for all $k<d$.

We review some of the theory of cellular spanning trees; for a complete treatment, see~\cite{Chapter}.

\begin{definition} \label{defn:tree}
Let $\Delta$ be a $d$-dimensional APC cell complex.  A subcomplex $\Upsilon \subseteq \Delta$ is a \emph{(cellular) spanning tree} if 
\begin{enumerate}
\item $\Upsilon^{(d-1)} = \Delta^{(d-1)}$ (``spanning'');
\item $\tilde H_{d-1}(\Upsilon;\mathbb{Z})$ is a finite group (``connected'');
\item $\tilde H_d(\Upsilon;\mathbb{Z})=0$ (``acyclic'');
\item $f_d(\Upsilon) = f_d(\Delta) - \betti_d(\Delta)$ (``count'').
\end{enumerate}
In fact, if $\Upsilon$ satisfies (a), then any two of conditions (b), (c), (d) together imply the third.  A \emph{$k$-tree} of~$\Delta$ is a cellular spanning tree of the $k$-skeleton~$\Delta^{(k)}$.  The set of $k$-trees of $\Delta$ is denoted by $\SST_k(\Delta)$.
\end{definition}
In particular, $\SST_{-1}(\Delta)=\{\0\}$, while $\SST_0(\Delta)=\Delta_0$ is the set of vertices of~$\Delta$ and $\SST_1(\Delta)$ is the set of spanning trees of the graph $\Delta^{(1)}$.

It is possible to relax the assumption that $\Delta$ is APC and give a more general definition of a cellular spanning forest (see~\cite{Chapter}), but since all the complexes we are considering here are APC, the assumption is harmless and will simplify the exposition.

For $k\leq d$, define
\[\tau_k =\tau_k(\Delta) = \sum_{\Upsilon\in\SST_k(\Delta)} \abs{\HH_{k-1}(\Upsilon;\Zz)}^2.\]
This number is the higher-dimensional analogue of the number of spanning trees of a graph.
Note that $\tau_{-1}=1$ and $\tau_0$ is the number of vertices.

We recall two classical results about tree counts in simplicial complexes.  First, let $n,d$ be integers and let $K_n^d$ denote the $d$-skeleton of a simplex with $n$-vertices.  Kalai~\cite{Kalai} proved that
\begin{equation}
\label{kalai-unweighted}
\tau_d(K_n^d)=n^{\binom{n-2}{d}},
\end{equation}
generalizing Cayley's formula $n^{n-2}$ for the number of labeled trees on~$n$ vertices.  Second, let $V_1,\dots,V_r$ be pairwise-disjoint sets of cardinalities $n_1,\ldots,n_r$, and let $\Delta_{n_1,\ldots,n_r}$ denote the \emph{complete colorful complex} on $V_1\cup\cdots\cup V_r$, whose facets are the sets meeting each $V_i$ in one point.  Adin~\cite{Adin} proved that
\begin{equation} \label{adin-unweighted}
\tau_k(\Delta_{n_1,\ldots,n_r}) =
\prod_{\substack{D \subseteq [r] \\ \abs{D}\leq k}} 
\left(\sum_{q\notin D}n_q\right)^{\binom{r-2-\abs{D}}{k-\abs{D}}\prod_{t\in D}(n_t-1)}
\end{equation}
for all $n_1,\dots,n_r$ and all $k=0,\dots,r-1$.  The case $r=1$ is trivial, and the case $r=2$ reduces to the known formula $\tau(K_{n,m})= n^{m-1}m^{n-1}$.  Furthermore, setting $r=n$, $k=d$, and $n_i=1$ for all $i$ recovers Kalai's formula~\eqref{kalai-unweighted}.

We now turn to weighted enumeration of spanning trees.  Let $\{X_\sigma\st \sigma\in\Delta\}$ be a family of commuting indeterminates corresponding to the faces of $\Delta$.  We define a polynomial analogue of $\tau_k(\Delta)$, the \emph{weighted tree count}, by
\[
\hat{\tau}_k = \hat{\tau}_k(\Delta) = \sum_{\Upsilon\in\SST_k(\Delta)} \abs{\HH_{k-1}(\Upsilon;\Zz)}^2 \prod_{\sigma\in\Upsilon_k} X_\sigma.
\]
Note that setting $X_\sigma=1$ for all $\sigma$ specializes $\hat{\tau}_k$ to $\tau_k$.  For the $d$-skeleton $K_n^d$ of an $n$-vertex simplex and with $X_\sigma = \prod_{i \in \sigma} y_i$,
Kalai~\cite{Kalai} proved that
\begin{equation}\label{Kalai-weight}
\hat{\tau}_d(K_n^d)=(X_{1}\cdots X_{n})^{\binom{n-2}{d-1}} (X_{1}+\cdots+X_{n})^{\binom{n-2}{d}}.
\end{equation}

Next we define weighted versions of the boundary operators of a cell complex.  The payoff will be a version of the cellular matrix-tree theorem that can be used recursively to calculate the weighted tree counts $\hat{\tau}_k(\Delta)$ for families such as complete colorful simplicial complexes.

Let $\xx=\{x_\sigma\st \sigma\in\Delta\}$ be a family of commuting indeterminates, and let $R=\Zz[\xx,\xx^{-1}]$ be the ring of Laurent polynomials in $\{x_\sigma\}$ over~$\Zz$.  In addition, set $X_\sigma=x_\sigma^2$.  For $k\leq d$, let $D_k$ denote the diagonal matrix with entries $(x_\sigma:\ \sigma\in\Delta_k)$.

\begin{definition} \label{weighted-boundary}
Let $\bd_k=\bd_k(\Delta)$ be the cellular boundary map $C_k(\Delta)\to C_{k-1}(\Delta)$, regarded as a matrix with rows and columns indexed by $\Delta_{k-1}$ and $\Delta_k$ respectively.  The \emph{$k^{th}$ weighted boundary map} is then the map $\wbd_k=\wbd_k(\Delta): C_k(\Delta,R)\to C_{k-1}(\Delta,R)$ given by $\wbd_k = D_{k-1}^{-1} \bd_k D_k$. 
\end{definition}

We will abbreviate $\wbd_k(\Delta)$ by $\wbd_k$ when the complex $\Delta$ is clear from context (and similarly for the other invariants to be defined shortly).  We define coboundary maps $\wcbd_k \colon C_{k-1}(\Delta,R) \to C_k(\Delta,R)$ by associating cochains with chains in the natural way, so that the matrix representing $\wcbd_k$ is just the transpose of that representing $\wbd_k$.

\begin{definition} \label{defn:weighted-L}
The $k$-th \emph{weighted up-down, down-up and total Laplacian operators} are respectively $\Lwud_k=\wbd_{k+1}\wcbd_{k+1}$, $\Lwdu_k = \wcbd_k \wbd_k$ and $\Lwtot_k = \Lwud_k + \Lwdu_k$. 
\end{definition}

Let $\sud_k$ and $\Swud_k$ denote the multisets of eigenvalues of~$\Lud_k$ and~$\Lwud_k$, respectively, and define $\sdu_k$, $\stot_k$, $\Swdu_k$, and $\Swtot_k$ similarly.  The notation $\mathbf{a} \circeq \mathbf{b}$ will indicate that two multisets $\mathbf{a}$ and $\mathbf{b}$ are equal up to the multiplicity of the element 0.  We recall the following well-known facts (e.g., \cite[Section 3]{Duval-Reiner}), and include their proofs for completeness.

\begin{proposition} \label{spectrum-identities}
For all $k\geq 0$ we have $\Swud_k \circeq \Swdu_{k+1}$ and $\Swtot_k \circeq \Swud_k \cup \Swdu_k$.  (Recall that multiset union is defined by adding multiplicities).
\end{proposition}

\begin{proof}
For the first identity, if $M$ is any square matrix and $v$ is a nonzero eigenvector of $M^TM$ with eigenvalue $\lambda\neq0$, then $Mv$ is a nonzero eigenvector of $MM^T$ with eigenvalue $\lambda$.  Thus the multiplicity of every nonzero eigenvalue is the same for $MM^T$ as for $M^TM$.  For the second identity, the operators
$\Lwdu_k$ and $\Lwud_k$ annihilate each other (because $\wbd_k\wbd_{k+1}=0$), so for every $\lambda\neq0$, the $\lambda$-eigenspace of $\Lwtot_k$ is the direct sum of those of $\Lwud_k$ and $\Lwdu_k$.
\end{proof}

The \emph{pseudodeterminant} $\pdet M$ of a square matrix $M$ is the product of its nonzero eigenvalues.  (Thus $\pdet M=\det M$ if $M$ is nonsingular.)  Define
\begin{equation} \label{define-pik}
\pi_k(\Delta) = \pdet \Lud_{k-1} = \pdet \Ldu_k, \qquad
\hat{\pi}_k(\Delta) = \pdet \Lwud_{k-1} = \pdet \Lwdu_k
\end{equation}
where the second and fourth equalities follow from Proposition~\ref{spectrum-identities}.  These invariants are linked to cellular tree and forest enumeration in~$\Delta$.

\section{Proof of the main formula} \label{section:proof}

In this section, we prove the weighted version of the cellular matrix-tree theorem (Theorem~\ref{thm:CMTT}) that we will use to enumerate trees in complete colorful complexes and skeleta of hypercubes.  As mentioned before, the result is a slight generalization of Theorem~5.3 of~\cite{MMRW}.

As before, let $\Delta$ be a finite cell complex of dimension~$d$.  Let $T\subseteq\Delta_d$ and $S\subseteq\Delta_{d-1}$ such that $\abs{T}=\abs{S}=f_d-\betti_d$.  Define
\[
\Delta_T = T \cup \Delta^{(d-1)},\qquad
\bar{S} = \Delta_{d-1}\setminus S,\qquad
\Delta_{\bar{S}} = \bar{S} \cup \Delta_{(d-2)},
\]
and let $\bd_{S,T}$ be the square submatrix of~$\bd_d$ with rows indexed by~$S$ and columns indexed by~$T$.

\begin{proposition}[{\cite[Proposition~2.6]{DKM2}}]  \label{nonsingular-criterion}
The matrix $\bd_{S,T}$ is nonsingular if and only if $\Delta_T\in\SST_d(\Delta)$ and 
$\Delta_{\bar{S}}\in \SST_{d-1}(\Delta)$.
\end{proposition} 

\begin{proposition}[{\cite[Proposition~2.7]{DKM2}}] \label{detD-formula}
If $\bd_{S,T}$ is nonsingular, then
\[\abs{\det \bd_{S,T}} = 
\frac{\abs{\HH_{d-1}(\Delta_T)} \cdot \abs{\HH_{d-2}(\Delta_{\bar{S}})}}{ \abs{\HH_{d-2}(\Delta)}}.\]
\end{proposition}

Note that $\det\wbd_{S,T} = (x_T/x_S) \det \bd_{S,T}$, so Propositions~\ref{nonsingular-criterion} and~\ref{detD-formula} immediately imply the following result.

\begin{proposition} \label{weighted-tools}
Let $T\subseteq\Delta_d$ and $S\subseteq\Delta_{d-1}$, with $\abs{T}=\abs{S}=f_d-\betti_d$.  
Then $\det\wbd_{S,T}$ is nonzero
if and only if $\Delta_T\in\SST_d(\Delta)$ and $\Delta_{\bar S}\in\SST_{d-1}(\Delta)$.  In that case,
\[
\pm\det \wbd_{S,T} 
= \frac{\abs{\HH_{d-1}(\Delta_T)} \cdot \abs{\HH_{d-2}(\Delta_{\bar S})}}{ \abs{\HH_{d-2}(\Delta)}}\cdot\frac{x_T}{x_S}.
\]
\end{proposition}

With these tools in hand, we can now prove Theorem~\ref{thm:CMTT}, following~\cite[Theorems~2.8(1) and~2.12(1)]{DKM2}.

\begin{proof}[Proof of Theorem~\ref{thm:CMTT}]
It is enough to prove the case $k=d$; the general case then follows by replacing~$\Delta$ with its $k$-skeleton $\Delta^{(k)}$, which is also APC and satisfies $\hat{\tau}_j(\Delta^{(k)})=\hat{\tau}_j(\Delta)$ for all $j\leq k$.  By the Binet-Cauchy formula and Proposition~\ref{weighted-tools}, we have
\begin{align*}
\hat{\pi}_d(\Delta)
~&= \sum_{T:\Delta_T\in\SST_d(\Delta)}\ \ \sum_{S:\Delta_{\bar{S}}\in\SST_{d-1}(\Delta)} (\det \wbd_{S,T})^2\\
&= \sum_{\Delta_T\in\SST_d}\ \ \sum_{\Delta_{\bar{S}}\in\SST_{d-1}}
\left(\frac{\abs{\HH_{d-1}(\Delta_T)}\cdot\abs{\HH_{d-2}(\Delta_{\bar{S}})}}{\abs{\HH_{d-2}(\Delta)}}\cdot\frac{x_T}{x_S}\right)^2\\
&= \frac{1}{\abs{\HH_{d-2}(\Delta)}^2\ X_{(d-1)}}
\left(\sum_{\Delta_T\in\SST_d} \abs{\HH_{d-1}(\Delta_T)}^2\; X_T\right) 
\left(\sum_{\Delta_{\bar{S}}\in\SST_{d-1}} \abs{\HH_{d-2}(\Delta_{\bar{S}})}^2\; X_{\bar{S}}\right)\\
&= \frac{\hat{\tau}_d(\Delta)\; \hat{\tau}_{d-1}(\Delta)}{\abs{\HH_{d-2}(\Delta)}^2\; X_{(d-1)}}.\qedhere
\end{align*}
\end{proof} 

The case $\HH_k(\Delta;\Zz)=0$ for all $k<\dim\Delta$ (so that the correction term $1/\abs{\HH_k(\Delta)}^2$ is trivial) is Theorem~5.3 of~\cite{MMRW}.

\begin{corollary}\label{improved-matrix-tree-thmforgraphs}
Let $G$ be a connected graph on vertex set $[n]$.  Then
\[\hat{\tau}(G)=\frac{X_1\cdots X_n}{X_1+\cdots+X_n}\hat{\pi}_1(G)=\frac{X_1\cdots X_n}{X_1+\cdots+X_n}\pdet\Lwud_0(G).\]
\end{corollary}

\begin{proof}
We may regard $G$ as a cell complex with $d=1$.  The 0-trees are just the vertices, so $\hat{\tau}_{d-1}(G)=X_1+\cdots+X_n$ and $X_{(d-1)}=X_1\cdots X_n$.  Moreover, $\HH_{-1}(G)=0$.  Now solving for $\hat{\tau}_d$ in Theorem~\ref{thm:CMTT} yields the desired formula.
\end{proof}

\section{Complete colorful complexes} \label{section:ccc}

In order to enumerate trees in complete colorful complexes, we will exploit the fact that these complexes are precisely the joins of 0-dimensional complexes.  Accordingly, we begin by recalling some facts about the join operation and its effect on Laplacian spectra.

Let $\Delta_1,\dots,\Delta_r$ be simplicial complexes on pairwise disjoint vertex sets $V_1,\dots,V_r$.  Their \emph{join} is the simplicial complex on vertex set $V_1 \disun\cdots\disun V_r$ given by
\[\Delta ~=~ \Delta_1 * \cdots * \Delta_r ~=~ \{\sigma = \sigma_1 \disun\cdots\disun \sigma_r\st \sigma_1 \in \Delta_1,\; \dots,\; \sigma_r \in \Delta_r\}.\]
The join of two shellable complexes is shellable by \cite[Corollary 2.9]{dress}.  In addition, Duval and Reiner~\cite[Theorem~4.10]{Duval-Reiner} showed that if $\Delta=\Delta_1*\Delta_2$, then
\[\stot_k(\Delta) = \multiset{\lambda_1+\lambda_2 \st 
\lambda_1\in \stot_{k_1}(\Delta_1),\ \lambda_2\in \stot_{k_2}(\Delta_2),\ k_1+k_2=k-1}.\]
This result may be obtained by observing that the simplicial chain complex $C_\bullet(\Delta;\Zz)$ can be identified with $C_\bullet(\Delta_1;\Zz)\otimes C_\bullet(\Delta_2;\Zz)$ (with an appropriate shift in homological degree, since $\dim(\sigma_1\cup \sigma_2)=\dim \sigma_1+\dim \sigma_2+1$).  This identification extends easily to $r$-fold joins, and it carries over to  weighted chain complexes (see Definition~\ref{weighted-boundary}) provided that the faces of $\Delta_1 * \cdots * \Delta_r $ are weighted by $x_{\sigma_1\disun\cdots\disun \sigma_r} = x_{\sigma_1} \cdots x_{\sigma_r}$. That is,
\[C_\bullet(\Delta;R_1\otimes_\Zz\cdots\otimes_\Zz R_r)=
C_\bullet(\Delta_1;R_1)\otimes\cdots\otimes C_\bullet(\Delta_r;R_r)\]
where $R_i$ is the ring of Laurent polynomials in the weights of faces of $\Delta_i$.  As a consequence, we obtain the following general formula for the weighted Laplacian spectrum of a join:
\begin{equation} \label{weighted-eigenvalues-of-join}
\Swtot_k(\Delta) = \bigmultiset{\lambda_1+\cdots+\lambda_r \st \lambda_q\in\Swtot_{k_q}(\Delta_q) \ \forall q, \ \ \textstyle\sum_{q=1}^r k_q = k-r+1}.
\end{equation}

We now introduce complete colorful complexes (or ``complete multipartite complexes''), which were Bolker's original motivation for introducing simplicial spanning trees in~\cite{Bolker}  and were studied in detail by Adin~\cite{Adin}.

\begin{definition} \label{defn:ccc}
Let $V_1=\{v_{1,1},\dots,v_{1,n_1}\},\dots,V_r=\{v_{r,1},\dots,v_{r,n_r}\}$ be finite sets of vertices, and let $\bar K_{V_q}$ denote the edgeless graph (= 0-dimensional complex) on vertex set $V_q$.  The \emph{complete colorful complex} $\Delta=\Delta_{n_1,\ldots,n_r}$ is the simplicial join $\bar K_{V_1}*\cdots*\bar K_{V_r}$ on vertex set $V=V_1\cup\cdots\cup V_r$.  We regard the indices $1,\dots,r$ as colors, so that the facets (resp., faces) of $\Delta$ are the sets of vertices having exactly one (resp., at most one) vertex of each color.
\end{definition}

The simplicial complex $\Delta$ is shellable because it is the join of shellable complexes.  In particular, the homology groups $\HH_k(\Delta;\Zz)$ vanish for all $k<\dim \Delta=r-1$.  Note that $\Delta$ is also the clique complex of the complete multipartite graph $K_{n_1,\dots,n_r}$. 

For the rest of the section, we fix the notation of Definition~\ref{defn:ccc}.  Let $\{x_{q,i} \st q\in[r],\ i\in[n_q]\}$ be a family of commuting indeterminates corresponding to the vertices of $\Delta$, and set $x_\sigma=\prod_{v_{q,i}\in \sigma} x_{q,i}$ and $X_\sigma=x_\sigma^2$.  Also, for each $q\in[r]$, define
\begin{align*}
S_q&=\sum_{i=1}^{n_q}X_{q,i}, &
P_q&=\prod_{i=1}^{n_q}X_{q,i}.
\end{align*}

\begin{lemma}\label{colorful-eigenvalues}
The eigenvalues of $\Lwtot_k(\Delta)$ are all of the form $\lambda_J=\sum_{q\notin J}S_q$, where $J$ ranges over all subsets of $[r]$ of size at most $k+1$.  The multiplicity of each $\lambda_J$ is $\binom{r-\abs{J}}{k+1-\abs{J}}\prod_{t\in J}(n_t-1)$.
\end{lemma}

\begin{proof}
For $q\in[r]$, let $\Delta_q$ be the 0-dimensional complex with vertices $V_q$.  Its weighted coboundary $\wcbd_0(\Delta_q)$ is represented by the column vector $U_q=(x_{q,1},\;\dots,\;x_{q,n_q})$.  Thus
\begin{align*}
\Lwtot_{-1}(\Delta_q) &= \Lwud_{-1}(\Delta_q) = U_q^T U_q, &
\Lwtot_0(\Delta_q) &= \Lwdu_0(\Delta_q) = U_q U_q^T.
\end{align*}
The first of these is the $1\x 1$ matrix with entry $S_q$.  The second is a $n_q\x n_q$ matrix of rank $1$; its nonzero eigenspace is spanned by $U_q$, with eigenvalue $U_q^T U_q = S_q$.  Hence
\begin{align*}
\Swtot_{-1}(\Delta_q) &= \multiset{S_q},  &
\Swtot_0(\Delta_q) &= \multiset{S_q;\ 0\colon n_q-1}
\end{align*}
and applying~\eqref{weighted-eigenvalues-of-join} gives
\[\Swtot_k(\Delta) = \bigmultiset{
\lambda_1 + \cdots + \lambda_r \st \ \
\begin{array}{l}
\lambda_q\in \Swtot_{k_q}(\Delta_q) \ \ \forall q,\\
k_q \in \{-1,0\}\ \ \ \,\forall q,\\
\lvert\{q\st k_q=0\}\rvert=k+1
\end{array} }.
\]
In particular, the eigenvalues of~$\Lwtot_k(\Delta)$ are all of the form $\lambda_J$.  Their multiplicities are as claimed because each instance of $\lambda_J$ arises by choosing one of the $n_q-1$ copies of the zero eigenvalue from $\Swtot_0(\Delta_q)$ for each $q\in J$, and then choosing the remaining $k+1-\abs{J}$ indices $q$ for which $k_q=0$.
\end{proof}

For $k\in\Zz$ and $1 \leq q \leq r$,  define 
\begin{equation} \label{Ekq}
E_{k,q}=\sum_{\substack{J\subseteq [r]\sm\{q\}\\ \abs{J} \leq k-1}}(-1)^{k-1-\abs{J}}\prod_{t\in J}n_t.
\end{equation}
Note that each $E_{k,q}$ is symmetric in $n_1,\dots,\widehat{n_q},\dots,n_r$; specifically, it is the sum of all monomials of degree at most $k-1$ in the expansion of $(-1)^{r-k}\prod_{i\in [r]\sm\{q\}}(n_i-1)$.

We now restate and prove the main theorem of this section, introducing additional notation that will make the proof easier.

\medskip

\noindent {\bf Theorem~\ref{complete-colorful-theorem}.}
\emph{Let $\Delta=\Delta_{n_1,\dots,n_r}$ be a complete colorful complex,
where $r,n_1,\dots,n_r$ are positive integers.  Then for all $0 \leq k\leq r-1$, we have
\[
\hat{\tau}_k(\Delta) = A_kB_k
\]
where
\begin{align*}
A_k &= \prod_{q=1}^r P_q^{E_{k,q}}, & 
B_k &= \prod_{\substack{J\subseteq [r]\\ \abs{J}\leq k}}
  \Big(\sum_{q \notin J} S_q \Big)^{\binom{r-2-\abs{J}}{k-\abs{J}} \prod_{t \in J}(n_t - 1)}.
\end{align*}}
\medskip

Note that when $k=r-1$, the expression for $B_k$ includes the nonstandard binomial coefficient $\binom{r-2-\abs{J}}{r-1-\abs{J}}$ (recall our conventions on binomial coefficients from~\eqref{binomial-coeffs}).

\begin{proof}
The proof proceeds by induction on $k$; the base cases are $k=-1$ and $k=0$.

{\bf Base cases:}\  When $k=-1$, we have $\hat{\tau}_{-1}=1$ and $A_{-1}=B_{-1}=1$ (adopting the standard convention that an empty product equals 1).  When $k=0$, we have $E_{0,q}=0$, so $A_0=1$ and $B_0=\sum_{q=1}^r S_q = \sum_{v_{q,j}\in V}X_{q,j}$, which equals $\hat{\tau}_0(\Delta)$ because the 0-trees of a complex are just its vertices. 

{\bf Inductive step:}\ Let $k\geq 0$.  Then Theorem~\ref{thm:CMTT} implies that
\begin{equation} \label{from-CMTT}
\hat{\tau}_{k+1}  \hat{\tau}_{k}^2 \hat{\tau}_{k-1} = X_{(k)}X_{(k-1)}\hat{\pi}_{k+1}\hat{\pi}_k
\end{equation}
(recall that all torsion factors equal 1 since $\Delta$ is shellable).  In order to show that $\hat{\tau}_{k+1}=A_{k+1}B_{k+1}$, it suffices to show that $A_kB_k$ satisfies this same recurrence, i.e., that
\begin{equation} \label{t-recurrence}
(A_{k+1}  A_k^2 A_{k-1})(B_{k+1}  B_k^2 B_{k-1})
~=~ X_{(k)}X_{(k-1)}\hat{\pi}_{k+1}\hat{\pi}_k.
\end{equation}

First, define
\[
F_{k,q}=\sum_{\substack{J\subseteq [r]\sm\{q\}\\ \abs{J}=k}} \ \prod_{t\in J} n_t
\] 
so that $X_{(k)}=\prod_{q=1}^r P_q^{F_{k,q}}$. 
Moreover, for $1 \leq k \leq r-1$ and $1 \leq q \leq r$, we have
$E_{k,q} + E_{k-1,q} = F_{k-1,q}$.
Then the definition of $A_k$ implies that
\begin{equation} \label{Ak-result}
A_{k+1} A_k^2 A_{k-1} 
= \prod_{q=1}^r P_q^{E_{k+1,q} + 2E_{k,q} + E_{k-1,q}} 
= \prod_{q=1}^r P_q^{F_{k,q} + F_{k-1,q}}
= X_{(k)} X_{(k-1)}.
\end{equation}

Second, for $J\subseteq[r]$, define $\lambda_J=\sum_{q\notin J}S_q$ (as in Lemma~\ref{colorful-eigenvalues}) and $\mu_J = \lambda_J^{\prod_{t \in J}(n_t-1)}$.
Then
\begin{align}
B_{k+1} B_k^2 B_{k-1} 
 &= \prod_{\substack{J\subseteq [r]\\ \abs{J}=j \leq k-1}} 
      \mu_J^{\binom{r - 2 - j}{k+1-j} + 2\binom{r - 2 - j}{k - j} + \binom{r - 2 - j}{k-1-j} } \notag\\
      &\qquad\qquad\x \prod_{\substack{J\subseteq [r]\\ \abs{J}=k}}
      \mu_J^{\binom{r-2-k}{k+1-k}+2\binom{r-2-k}{0}}
    \prod_{\substack{J\subseteq [r]\\ \abs{J}=k+1}} 
      \mu_J^{\binom{r-k-3}{0}} \notag\\
 &= \prod_{\substack{J\subseteq [r]\\ \abs{J}=j\leq k-1}} 
      \mu_J^{\binom{r - j}{k+1-j}} 
    \prod_{\substack{J\subseteq [r]\\ \abs{J}=k}} \mu_J^{r-k}
    \prod_{\substack{J\subseteq [r]\\ \abs{J}=k+1}} \mu_J \notag\\
 &= \prod_{\substack{J\subseteq [r]\\ \abs{J}=j \leq k+1}}
      \mu_J^{\binom{r - j}{k+1-j}} \notag\\
 &= \prod_{\substack{J\subseteq [r]\\ \abs{J}=j \leq k+1}} 
      \lambda_J^{\binom{r - j}{k+1-j}\prod_{t \in J}(n_t-1)} \notag\\
 &=\ \pdet \Lwtot_k(\Delta) ~=~ \hat{\pi}_k \hat{\pi}_{k+1} \label{Bk-result}
\end{align}
by Lemma~\ref{colorful-eigenvalues}, Proposition~\ref{spectrum-identities}, and the definition of $\hat{\pi}_k$ in~\eqref{define-pik}.  Now combining~\eqref{Ak-result} and~\eqref{Bk-result} yields the desired recurrence~\eqref{t-recurrence}.
\end{proof}

\begin{remark}
Kalai's formula~\eqref{Kalai-weight} now follows as a corollary, by regarding the simplex on~$n$ vertices as the complete colorful complex with $n$ vertices, all of different colors.
\end{remark}

\begin{remark}
Setting $k=1$ gives the degree-weighted tree count for the complete multipartite graph $K_{n_1,\dots,n_r}$.  Here $E_{1,q}=1$, so the theorem gives
\begin{align*}
\hat{\tau}_1=A_1B_1 &= \left(\prod_{q=1}^r P_q\right)
\left(\prod_{\substack{J\subseteq [r]\\ \abs{J}\leq 1}}
\lambda_J^{\binom{r-2-\abs{J}}{1-\abs{J}} \prod_{t \in J}(n_t - 1)}\right) \\
&= \left(\prod_{v_{q,j}\in V} X_{q,j}\right)
\left(\sum_{v_{q,j}\in V} X_{q,j}\right)^{r-2}
\left( \prod_{t=1}^r \Big(\sum_{q\neq t} S_q \Big)^{(n_t - 1)} \right),
\end{align*}
which recovers a theorem of Clark~\cite[Theorem~2]{Clark}.  Further specializing all the indeterminates to 1 and writing $n=n_1+\cdots+n_r$ yields the unweighted tree count
\[\tau_1(K_{n_1,\dots,n_r}) ~=~ n^{r-2}  \prod_{t=1}^r (n-n_t)^{(n_t - 1)},\]
a well-known formula that has been rediscovered many times; the earliest source appears to be Austin~\cite{Austin}.  
\end{remark}

\begin{remark}
Consider the special case $k=r-1$.  Here, we have
\[A_{r-1} = \prod_{q=1}^r P_q^{E_{r-1},q} ~=~ \prod_{q=1}^r P_q^{ \left(\prod_{t\neq q}n_t\right)-\left(\prod_{t\neq q}(n_t-1) \right)}.\]
To simplify $B_{r-1}$, note that our conventions on binomial coefficients imply that
\[\binom{r-2-j}{r-1-j}=\begin{cases} 0 & \text{ if } j<r-1,\\ 1 &\text{ if } j=r-1,\end{cases}\]
so that
\[B_{r-1} = \prod_{\substack{J\subset[r] \\ \abs{J}=r-1}} \lambda_J^{\prod_{t\in J}(n_t-1)}.\]
Thus, we find
\[\hat{\tau}_{r-1}(\Delta_{n_1,\ldots,n_r})=
\prod_{q=1}^r\bigg(P_q^
{\prod_{i\neq q}n_i-\prod_{i\neq q}(n_i-1)}\bigg)
\bigg(S_q^{\prod_{i\neq q}(n_i-1)}\bigg).
\]
\end{remark}

\begin{example} \label{exa:octahedron}
The complete colorful complex $\Delta_{2,2,2}$ is the boundary of an octahedron.  By Theorem~\ref{complete-colorful-theorem}, its degree-weighted tree counts are
\begin{align*}
\hat{\tau}_2 &= (X_{1,1}X_{1,2}X_{2,1}X_{2,2}X_{3,1}X_{3,2})^3 S_1 S_2 S_3,\\
\hat{\tau}_1 &= X_{1,1}X_{1,2}X_{2,1}X_{2,2}X_{3,1}X_{3,2} (S_2 + S_3) (S_1 + S_3) (S_1 + S_2) (S_1 + S_2 + S_3),
\end{align*}
where $S_q=X_{q,1} + X_{q,2}$.
\end{example}

\begin{example}
Let $K_n$ be the complete graph on $n$ vertices $x_1,\dots,x_n$, and let $\bar K_2$ be the empty graph on 2 vertices $y_1, y_2$.  Their join $\Delta=K_n*\bar K_2$ is a 2-dimensional simplicial complex;  topologically, it is the suspension of $K_n$.  While $\Delta$ is not a skeleton of any complete colorful complex, we can still use the formula~\eqref{weighted-eigenvalues-of-join} to compute its weighted Laplacian eigenvalues and thus enumerate its trees.  Of course, it is possible to extract the weighted tree counts $\hat{\tau}_k$ directly from the spectra of the total Laplacians, since $\hat{\tau}_{k+1}  \hat{\tau}_{k}^2 \hat{\tau}_{k-1} = X_{(k)}X_{(k-1)}\pdet \Lwtot_k(\Delta)$ by \eqref{from-CMTT} and~\eqref{Bk-result}, but in this case it is a little more convenient to work with the up-down Laplacian spectra.

Let $S_X = X_1+\cdots+X_n$ and $S_Y=Y_1+Y_2$.  It is easy to verify (e.g., by Lemma~\ref{colorful-eigenvalues}) that
\begin{align*}
\Swtot_{-1}(\bar K_2) &= \multiset{S_Y},		& \Swtot_{-1}(K_n) &= \multiset{S_X},\\
\Swtot_{0}(\bar K_2) &= \multiset{0;\ S_Y},	& \Swtot_0(K_n) &= \multiset{S_X\colon n},\\
&								& \Swtot_1(K_n) &= \multiset{S_X\colon n-1;\ 0\colon \tbinom{n}{2}-(n-1)}.
\end{align*}
Noting that $\binom{n}{2}-(n-1)=\binom{n-1}{2}$ and applying the join formula~\eqref{weighted-eigenvalues-of-join}, we get
\begin{align*}
\Swtot_{-1}(\Delta) &= \multiset{S_X + S_Y},\\
\Swtot_0(\Delta) &= \multiset{S_X + S_Y\colon n+1;\ S_X},\\
\Swtot_1(\Delta) &= \multiset{S_Y\colon \tbinom{n-1}{2};\ S_X + S_Y\colon 2n-1;\ S_X\colon n},\\
\Swtot_2(\Delta) &= \multiset{0\colon \tbinom{n-1}{2};\ S_X\colon n-1;\ S_Y\colon \tbinom{n-1}{2};\ S_X + S_Y\colon n-1},
\end{align*}
and thus, by Proposition~\ref{spectrum-identities}, we have
\begin{align*}
\Swud_{-1}(\Delta) &= \multiset{S_X + S_Y},\\
\Swud_0(\Delta) &\circeq \multiset{S_X + S_Y\colon n;\ S_X},\\
\Swud_1(\Delta) &\circeq \multiset{S_Y\colon \tbinom{n-1}{2};\ S_X + S_Y\colon n-1;\ S_X\colon n-1},
\end{align*}
and therefore
\[
\hat{\pi}_0(\Delta) = S_X + S_Y,\quad \hat{\pi}_1(\Delta) = S_X(S_X + S_Y)^n,\quad
\hat{\pi}_2(\Delta) = S_X^{n-1} S_Y^{\tbinom{n-1}{2}} (S_X + S_Y)^{n-1}.
\]

We can now write down the degree-weighted tree counts $\hat{\tau}_k(\Delta)$.  The $0$-dimensional trees are vertices, so $\hat{\tau}_0(\Delta) = S_X+S_Y = X_1 + \cdots + X_n + Y_1 + Y_2$, and then applying Theorem~\ref{thm:CMTT} gives
\begin{align*}
\hat{\tau}_1 &= (X_1 \cdots X_n) (Y_1 Y_2) (X_1 + \cdots + X_n + Y_1 + Y_2)^{n-1} (X_1 + \cdots + X_n),\\
\hat{\tau}_2 &= (X_1 \cdots X_n)^n (Y_1 Y_2)^{n-1}(X_1 + \cdots + X_n)^{n-2} (Y_1 + Y_2)^{\tbinom{n-1}{2}}.
\end{align*}

An open question is to simultaneously generalize this example and Theorem~\ref{complete-colorful-theorem} by giving a formula for weighted Laplacian spectra of arbitrary joins of skeleta of simplices, i.e., complexes of the form 
$K_{n_1}^{d_1}*\cdots*K_{n_r}^{d_r}$.
\end{example}

\section{Hypercubes} \label{section:hypercubes}

In this section we prove a formula (Theorem~\ref{thm:cube-tree}) for weighted enumeration of trees in the hypercube~$Q_n$.  This formula was proposed in~\cite[Conjecture~4.3]{DKM2}, and the special case $k=1$ (enumerating trees in hypercube graphs) is equivalent to~\cite[Theorem~3]{JLM-VR-Facto}.  We begin by describing the cell structure of~$Q_n$.

Let $n$ be a nonnegative integer.  The \emph{hypercube} $Q_n$ is the topological space $[0,1]^n\subset\Rr^n$ made into a regular cell complex with $3^n$ cells, each identified with an $n$-tuple $\sigma=(\sigma_1,\ldots,\sigma_n)$, where each $\sigma_i\in\{[0,1],0,1\}$.  Thus $\dim\sigma=\abs{\{i\st \sigma_i=[0,1]\}}$, and the $f$-polynomial of~$Q_n$ is
\[f(Q_n;t) = \sum_{\sigma\in Q_n} t^{\dim\sigma} = (t+2)^n.\]
Since $Q_n$ is contractible, for every $k<d\leq n$, we have
\begin{equation} \label{Qn-no-torsion}
\HH_k(Q_{n,d};\Zz)=\HH_k(Q_n;\Zz)=0
\end{equation}
where $Q_{n,d}$ denotes the $d$-skeleton of $Q_n$.

Let $q_1,\dots,q_n,y_1,\dots,y_n,z_1,\dots,z_n$ be commuting indeterminates.  For $A\subseteq[n]$, let $q_A=\prod_{j\in A} q_j$, and define $y_A$ and $z_A$ similarly.  Assign each face $\sigma\in Q_n$ the weight $X_\sigma=q_Ay_Bz_C$, where $A=\{i\st \sigma_i=[0,1]\}$, $B=\{i\st \sigma_i=0\}$, and $C=\{i\st \sigma_i=1\}$.

We will need the following simple combinatorial identity.

\begin{lemma} \label{standardized-identity}
Let $N,K$ be integers with $N\geq 0$.  Then
\begin{equation} \label{binomial-identity-eqn}
\sum_{i=K}^{N}\binom{i}{K}\binom{N}{i}~=~\binom{N}{K} 2^{N-K}.
\end{equation}
\end{lemma}

\begin{proof}
Both sides count the number of triples $(A,B,C)$ with $A\disun B\disun C=[N]$ and $\abs{A}=K$.  The index of summation $i$ represents $\abs{A\cup B}$.  Note that both sides vanish if $K<0$ or $K>N$.
\end{proof}

We now restate and prove the main theorem on hypercubes.
\medskip

\noindent {\bf Theorem~\ref{thm:cube-tree}.}
\emph{Let $n\geq k\geq 0$.  Then
\[
\hat{\tau}_k(Q_n)=q_{[n]}^{\sum_{i=k-1}^{n-1}\binom{n-1}{i}\binom{i-1}{k-2}}
\prod_{\substack{S\subseteq[n]\\ \abs{S}>k}}
\left(\sum_{i\in S} q_i(y_i+z_i)\prod_{j\in S\sm i}y_jz_j\right)^{\binom{\abs{S}-2}{k-1}}.
\]}

\begin{proof}
If $n=k=0$, then the theorem reduces to the statement $\hat{\tau}_0(Q_0)=1$, which is clear.  Henceforth, we assume $n>0$.

First we rewrite the theorem in a more convenient form.  For $S\subseteq[n]$, let
\[u_S=\sum_{i \in S} q_i\left(\frac{1}{y_i}+\frac{1}{z_i}\right)\]
so that
\[\sum_{i\in S}\Big(q_i(y_i+z_i)\prod_{j\in S\sm i}y_jz_j\Big)=u_Sy_Sz_S.\]
Thus, the theorem is equivalent to the statement that $\hat{\tau}_k(Q_n)=A_kB_kC_k$, where
\begin{align*}
A_k &= q_{[n]}^{\sum_{i=k-1}^{n-1}\binom{n-1}{i}\binom{i-1}{k-2}},&
B_k &= \prod_{\substack{S\subseteq[n]\\ \abs{S}>k}}u_S^{\binom{\abs{S}-2}{k-1}},&
C_k &= \prod_{\substack{S\subseteq[n]\\ \abs{S}>k}}\left(y_S z_S\right)_\PHX^{\binom{\abs{S}-2}{k-1}}.
\end{align*}

We proceed by induction on $k$.  The base case is $k=0$.  For the inductive step, we know by Theorem~\ref{thm:CMTT} that $\hat{\tau}_k\hat{\tau}_{k-1}=\hat{\pi}_{k} X_{(k-1)}$ for all $k>0$, where $X_{(k-1)}=\prod_{\dim \sigma = k-1} X_{\sigma}$, as before.  (Note that the torsion factor $\abs{\HH_{k-2}(\Delta)}^2$ vanishes by~\eqref{Qn-no-torsion}.)  We will show that  $A_kB_kC_k$ satisfies the same recurrence, i.e., that
\begin{equation} \label{want-t-recurrence}
A_kB_kC_kA_{k-1}B_{k-1}C_{k-1}=\hat{\pi}_{k}X_{(k-1)}.
\end{equation}

{\bf Base case:} Let $k=0$.  First, $\binom{i-1}{-2}=0$ for all $i>0$, and $\binom{n-1}{-1}=0$ for $n\geq 1$, so the only non-vanishing summand in the exponent of $A_0$ is the $i=0$ summand, namely $\binom{n-1}{0}\binom{-1}{-2}=-1$, giving $A_0 = (q_1\cdots q_n)^{-1}$.

Second, if $S\subseteq[n]$ is a nonempty set then $\binom{\abs{S}-2}{-1}=0$ unless $\abs{S}=1$.  So the surviving factors in $B_0$ and $C_0$ correspond to the singleton subsets of~$[n]$, giving
\[A_0B_0C_0 = q_{[n]}^{-1} \prod_{j\in[n]} q_j\left(\frac{1}{y_j}+\frac{1}{z_j}\right)y_jz_j = \prod_{j\in[n]} (y_j+z_j)\]
which is indeed $\hat{\tau}_0(Q_n)$  (= $\hat{\pi}_0(Q_n)$).

{\bf Inductive step:} We now assume that $k>0$.  First, the exponent on $q_{[n]}$ in $A_kA_{k-1}$ is
\begin{align}
\sum_{i=k-1}^{n-1}\binom{n-1}{i}&\binom{i-1}{k-2}+
\sum_{i=k-2}^{n-1}\binom{n-1}{i}\binom{i-1}{k-3} \notag\\
=~& \binom{n-1}{k-2}+\sum_{i=k-1}^{n-1}\binom{n-1}{i}\left[\binom{i-1}{k-2}+\binom{i-1}{k-3}\right] \notag\\
=~& \binom{n-1}{k-2}+\sum_{i=k-1}^{n-1}\binom{n-1}{i}\binom{i}{k-2} \notag\\
=~& \sum_{i=k-2}^{n-1}\binom{n-1}{i}\binom{i}{k-2}\notag\\
=~& \binom{n-1}{k-2} 2^{n-k+1} \label{A-result}
\end{align}
by Lemma~\ref{standardized-identity} with $N=n-1$ and $K=k-2$.  The second equality uses the Pascal recurrence.  Note that $(i,k)=(0,2)$, the one instance where the Pascal recurrence is invalid, does not occur.

Second, we calculate
\begin{align}
B_kB_{k-1} &= 
\prod_{\substack{S\subseteq[n]\\ \abs{S}>k}}u_S^{\binom{\abs{S}-2}{k-1}}\ \ 
\prod_{\substack{S\subseteq[n]\\ \abs{S}>k-1}}u_S^{\binom{\abs{S}-2}{k-2}}\notag\\
&=
\prod_{\substack{S\subseteq[n]\\ \abs{S}>k}}u_S^{\binom{\abs{S}-2}{k-1}+\binom{\abs{S}-2}{k-2}}\ \ 
\prod_{\substack{S\subseteq[n]\\ \abs{S}=k}}u_S\notag\\
&=
\prod_{\substack{S\subseteq[n]\\ \abs{S}>k}}u_S^{\binom{\abs{S}-1}{k-1}}\ \ 
\prod_{\substack{S\subseteq[n]\\ \abs{S}=k}}u_S\notag\\
&=
\prod_{\substack{S\subseteq[n]\\ \abs{S}\geq k}}u_S^{\binom{\abs{S}-1}{k-1}}. \label{B-result}
\end{align}
Again, the second equality uses the Pascal recurrence, and the exception $\abs{S}=k=1$ does not occur.

Third, we consider $C_kC_{k-1}$.  This is a monomial that is symmetric in the variables $y_1,\dots,y_n,z_1,\dots,z_n$, so it is sufficient to calculate the exponent of~$y_n$, which is
\begin{align}
\sum_{\substack{S\subseteq[n]:\\ n\in S,\ \abs{S}>k}}\!\!\binom{\abs{S}-2}{k-1} ~&+\!\!\!
\sum_{\substack{S\subseteq[n]:\\ n\in S,\ \abs{S}>k-1}}\binom{\abs{S}-2}{k-2} \notag\\
=~& \sum_{\substack{S\subseteq[n]:\\ n\in S,\ \abs{S}=k}}\binom{k-2}{k-2} +
\sum_{\substack{S\subseteq[n]:\\ n\in S,\ \abs{S}>k}}\left[\binom{\abs{S}-2}{k-2}+\binom{\abs{S}-2}{k-1}\right] \notag\\
=~& \binom{n-1}{k-1} +
\sum_{\substack{S\subseteq[n]:\\ n\in S,\ \abs{S}>k}}\binom{\abs{S}-1}{k-1} \notag\\
=~& \binom{n-1}{k-1} + \sum_{i=k}^{n-1}\binom{i}{k-1} \binom{n-1}{i} 
\qquad\text{(setting $i=\abs{S}-1$)}\notag\\
=~& \sum_{i=k-1}^{n-1}\binom{i}{k-1} \binom{n-1}{i} \notag\\
=~& \binom{n-1}{k-1} 2^{n-k} \label{C-result}
\end{align}
by Lemma~\ref{standardized-identity} with $N=n-1$ and $K=k-1$.  Here again we have used the Pascal recurrence in the third step, and the exception $\abs{S}=k=1$ does not occur.  Putting together~\eqref{A-result},~\eqref{B-result} and~\eqref{C-result} gives
\begin{align*}
A_kA_{k-1}B_kB_{k-1}C_kC_{k-1} &~=~
q_{[n]}^{\binom{n-1}{k-2} 2^{n-k+1}}
\prod_{\substack{S\subseteq[n]\\ \abs{S}\geq k}}u_S^{\binom{\abs{S}-1}{k-1}} \ \ 
(y_{[n]}z_{[n]})_\PHX^{\binom{n-1}{k-1} 2^{n-k}}.
\end{align*}

Now we consider the right-hand side of the desired equality~\eqref{want-t-recurrence}.   By~\cite[Theorem 4.2]{DKM2}, the eigenvalues of $\Lwud_{k-1}(Q_n)$ are $\{u_S\st S\subseteq[n],\ \abs{S}\geq k\}$, each occurring with multiplicity $\binom{\abs{S}-1}{k-1}$.  Therefore
\begin{equation} \label{pihat-k}
\hat{\pi}_k=\prod_{\substack{S \subseteq [n] \\ \abs{S}\geq k}} u_S^{\binom{\abs{S}-1}{k-1}}=B_kB_{k-1}.
\end{equation}
Meanwhile, by a simple combinatorial calculation (which we omit), we have
\begin{align}
X_{(k-1)}
~&=~ \prod_{\substack{A\disun B \disun C = [n]\\\abs{A}=k-1}}   q_A y_B z_C 
\ \ = \ \ 
q_{[n]}^{\binom{n-1}{k-2} 2^{n-k+1}} (y_{[n]}z_{[n]})^{\binom{n-1}{k-1} 2^{n-k}} \notag\\
~&=~ A_kA_{k-1}C_kC_{k-1}  \label{product-kfaces}
\end{align}
by~\eqref{A-result} and~\eqref{C-result}.  Now combining~\eqref{product-kfaces} with~\eqref{pihat-k} gives
\[\hat{\pi}_k X_{(k-1)} =  A_kA_{k-1}B_kB_{k-1}C_kC_{k-1}
\]
establishing~\eqref{want-t-recurrence} as desired.
\end{proof}

\section{Another tree count for hypercubes}\label{sec-twoinvariants}

In this section, we present a very different way of calculating the weighted tree counts $\hat{\tau}_k(Q_n)$ of a hypercube $Q_{n}$, building on the previous work of Kook and Lee~\cite{KL15}. This gives an interpretation of some of the exponents in the formula in terms of unsigned reduced Euler characteristics. Section~\ref{sec-E} introduces a combinatorial formula for the reduced Euler characteristic of the skeleton $Q_n^{(k)}$, and Section~\ref{sec-L} presents a logarithmic approach to obtain a formula for $\hat{\tau}_k(Q_n)$. 

\subsection{Reduced Euler characteristics for a hypercube}\label{sec-E}
For $k \in [0,n]$, let $(Q_n)_k$ denote the set of all $k$-cells in $Q_n$, and note that $\abs{(Q_n)_k}=2^{n-k}\binom{n}{k}$. Define $(Q_n)_{-1}$ to be a set with one element. Then we have 
\begin{equation} \label{euler-char-formula}
\tilde{\chi}{(Q_n^{(k)})}=\sum_{i=-1}^{k}{(-1)^{i}}{\abs{(Q_n)_{i}}}=\sum_{i=0}^{k}(-1)^{i}2^{n-i}\binom{n}{i}-1.
\end{equation}
Note that since $Q_{n}$ is acyclic, the reduced homology groups of $Q_{n}^{(k)}$ vanish except in dimension $k$, and 
\begin{equation} \label{euler-char-sign}
\abs{\tilde{\chi}({Q_{n}^{(k)}})}=(-1)^{k}\tilde{\chi}({Q_{n}^{(k)}})=\rk_{\Zz}{\tilde{H}_{k}(Q_{n}^{(k)})}.
\end{equation} 
We now establish a combinatorial interpretation of the \emph{unsigned} Euler characteristic $\abs{\tilde{\chi}{(Q_n^{(k)})}}$, as stated in the introduction. 
\medskip

\begin{proposition} \label{thm-Euler}
Let $k \in [0,n]$.  Then
\[\abs{\tilde{\chi}({Q_{n}^{(k)}})}=\sum_{i=k+1}^{n}{\binom{n}{i}\binom{i-1}{k} }\]
and consequently (combining with~\eqref{euler-char-formula})
\[
\sum_{i=0}^{k}(-1)^{i}2^{n-i}\binom{n}{i}-1=(-1)^{k}\sum_{i=k+1}^n\binom{n}{i}\binom{i-1}{k}.
\]
\end{proposition}

\begin{proof}
For $k \in [0,n]$, define $A_{n,k}$ to be the subcomplex of $Q_{n}^{(k)}$ generated by the $k$-cells $\sigma=(\sigma_1,\dots,\sigma_{n})$ with the following property: if $\sigma_j=1$ for any $j$, then there is some $i<j$ for which $\sigma_i=[0,1]$.  Then the number of $k$-cells in $Q_{n}^{(k)}\setminus A_{n,k}$ is 
$\sum_{i=k+1}^{n}{\binom{n}{i}\binom{i-1}{k}}$ (here $i$ represents the number of coordinates of~$\sigma$ that are not $0$; the leftmost of these must be~1).

We will show that $A_{n,k}$ is acyclic and contains $Q_{n}^{(k-1)}$, from which the formula for $\abs{\tilde{\chi}({Q_{n}^{(k)}})}$ will immediately follow.  We proceed by induction on $n$. Note that $A_{n,n}=Q_{n}^{(n)}=Q_n$ and $A_{n,0}=\{(0,\dots,0)\}$ have these properties for any $n$, which covers the base step. Now, note that $A_{n,k}$ can be defined recursively as follows (refer to Figure~\ref{A31A32}):
\[
A_{n,k}=B \cup C \cup D
\]
for $k \in [1,n-1]$, where $B=A_{n-1,k} \times 0$, $C=A_{n-1,k} \times 1$, and $D=A_{n-1,k-1} \times [0,1]$. Since all of $B$, $D$, and $B\cap D=A_{n-1,k-1}\times 0$ are acyclic by induction,  $B\cup D$ is also acyclic. Similarly, $C\cup D$ is acyclic. Since $(B\cup D)\cap(C\cup D)=D$ is acyclic, we conclude that $A_{n,k}$ is acyclic.  Also by induction, $B\cup C$ contains all $(k-1)$-cells in $Q_{n}$ whose last coordinate is $0$ or $1$, and $D$ contains those whose last coordinate is $[0,1]$. 
Hence, $A_{n,k}$ contains ${Q_{n}^{(k-1)}}$.
\end{proof}

\begin{figure}[h]
\centering
\raisebox{15pt}{(a)}\includegraphics[scale=0.35]{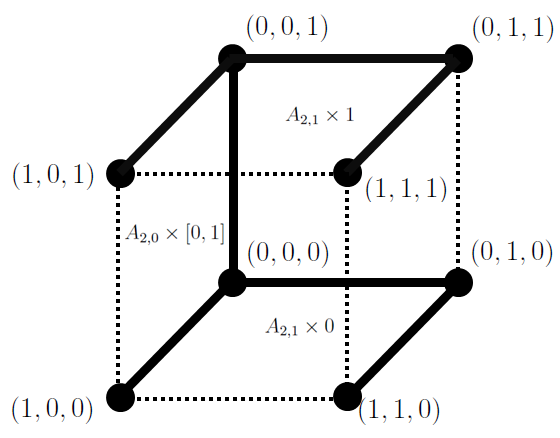}\hfill
\raisebox{15pt}{(b)}\includegraphics[scale=0.35]{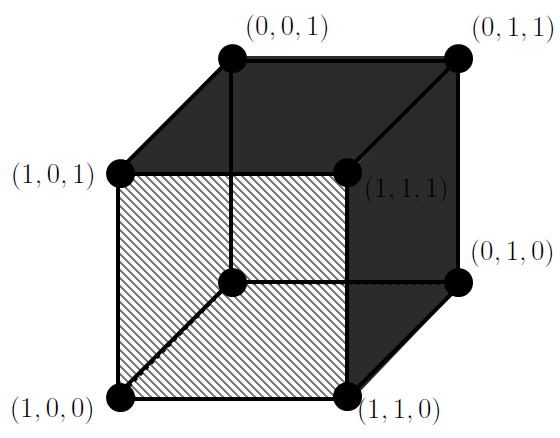}
\caption{(a) The complex generated by bold edges is $A_{3,1}$ and the set of dashed edges is $(Q_3)_{1}\setminus A_{3,1}$. (b) The complex generated by the five shaded $2$-cells is $A_{3,2}$, and the meshed $2$-cell is $(Q_3)_{2}\setminus A_{3,2}$. \label{A31A32}}
\end{figure}

\subsection{Interpretation of $\hat{\tau}_k(Q_n)$ via a logarithmic generating function} \label{sec-L}

We will give a logarithmic generating function for the weighted tree counts of the hypercube $Q_n$.  Logarithmic generating functions for unweighted tree counts of an acyclic cell complex were given in~\cite[Theorem~8]{KK}, and for weighted tree counts of a $\Zz$-APC simplicial complex in~\cite[Theorem~5]{KL15}.  In order to focus on the multiplicities of Laplacian eigenvalues, we will adopt a \emph{formal} logarithm satisfying the rule
\[
\log(ab)=\log a+\log b
\] 
where $a,b$ are nonzero elements of some field $\mathbb{F}$.

As in the case for complete colorful complexes \eqref{from-CMTT}, Theorem \ref{thm:CMTT} and Proposition \ref{spectrum-identities} imply the following identity for $k \in [-1,n]$:
\[
(\det \hat{L}^{\rm tot}_k)\left(\prod_{\sigma \in (Q_n)_{k-1}}{X_\sigma}\right)\left(\prod_{{\sigma'} \in (Q_n)_{k}}{X_{\sigma'}}\right)=  \hat{\tau}_{k-1}\hat{\tau}_k^2\hat{\tau}_{k+1}
\]
where $\hat{\tau}_{-2}=\hat{\tau}_{-1}=\hat{\tau}_{n+1}=1$. Since $Q_n$ is acyclic, we have $\det \hat{L}^{\rm tot}_k\ne 0$ for $k \in [-1, n]$ by combinatorial Hodge theory \cite{Fr}, and therefore $(\det \hat{L}^{\rm tot}_k) = \hat{\pi}_{k+1} \hat{\pi}_k$ by Proposition~\ref{spectrum-identities}.  Applying the formal logarithm to this identity, we obtain
\begin{equation}\label{logthm4.2}
\hat{\omega}_{k}+\log\left(\prod_{\sigma\in (Q_n)_{k-1}}X_{\sigma}\right) +\log\left(\prod_{\sigma'\in (Q_n)_{k}}X_{\sigma'}\right)=\hat{\kappa}_{k-1}+2\hat{\kappa}_{k}+\hat{\kappa}_{k+1}
\end{equation}
where $\hat{\omega}_{k}=\log\det \hat{L}^{\rm tot}_k$ and $\hat{\kappa}_{k}=\log\hat{\tau}_{k}$. Define 
\[
\hat{D}(x)=\sum_{k=-1}^{n}\hat{\omega}_k x^{k+1}\mbox{\quad and \quad}
\hat{K}(x)=\sum_{k=0}^{n}{\hat{\kappa}_k x^k}\, .
\]

We will establish a relationship between $\hat{D}(x)$ and $\hat{K}(x)$ as follows. Let $k\in[0,n]$ and fix $i \in [n]$. Define $\nu_{k}=\abs{\{\sigma \in (Q_n)_{k}\st \sigma_i=0\}}=\abs{\{\sigma \in (Q_n)_{k}\st \sigma_i=1\}}$ and $\epsilon_{k}=\abs{\{\sigma \in (Q_n)_{k}\st \sigma_i=[0,1]\}}$, so that
\[
\nu_{k}=\binom{n-1}{k}2^{n-1-k} \mbox{\quad and \quad} 
\epsilon_{k}=\binom{n-1}{k-1}2^{n-k}\,.
\]
Note that $\nu_{n}=\epsilon_{0}=0$.
Now define 
\[
F_{\nu}(x)=\sum_{k=0}^{n}{\nu_k x^{k+1}}\mbox{\quad and \quad}
F_{\epsilon}(x)=\sum_{k=0}^{n}{\epsilon_k x^{k+1}}\, .
\]

\begin{proposition}\label{genecube}  
We have
\begin{equation}\label{geneftncube}
\hat{K}(x)=(1+x)^{-2}\hat{D}(x)+(1+x)^{-1}\big(\log(y_{[n]}z_{[n]}) F_{\nu}(x)+\log(q_{[n]}) F_{\epsilon}(x)\big).
\end{equation}
\end{proposition}

\begin{proof} For $k\in[0,n]$, note that 
$\prod_{\sigma' \in (Q_n)_{k}}{X_{\sigma'}}=\prod_{i \in [n]}{(y_iz_i)}^{\nu_{k}}\prod_{i\in[n]}{q_i}^{\epsilon_{k}}$, which follows directly from the definitions of $\nu_{k}$ and $\epsilon_{k}$. Applying the formal logarithm to this identity, we get $\log(\prod_{\sigma' \in (Q_n)_{k}}{X_{\sigma'}})=\nu_{k}\log(y_{[n]}z_{[n]})+\epsilon_{k}\log(q_{[n]})$. Therefore, for $k\in[-1,n]$, we obtain from \eqref{logthm4.2}
\[
\hat{\kappa}_{k-1}+2\hat{\kappa}_{k}+\hat{\kappa}_{k+1}
=\hat{\omega}_{k}+(\nu_{k-1}+\nu_{k})\log(y_{[n]}z_{[n]})+(\epsilon_{k-1}+\epsilon_{k})\log(q_{[n]})
\]
where $\hat{\kappa}_{k}, \nu_{k}$, and $\epsilon_{k}$ are zero for $k\notin[0,n]$. Multiplying both sides by $x^{k+1}$ and summing over $k\in[-1,n+1]$ gives 
\[
(1+x)^{2}\hat{K}(x)=\hat{D}(x)+(1+x)(\log(y_{[n]}z_{[n]}) F_{\nu}(x)+\log(q_{[n]}) F_{\epsilon}(x))\, ,
\]
and the result follows. 
\end{proof}

We now restate the main result of this part of the paper.
\medskip

\noindent {\bf Theorem~\ref{inter-hyper-theorem}.}  \emph{For $k \in [n]$, the $k$-th weighted tree count of $Q_n$ is
\begin{equation}\label{formula-cube}
\hat{\tau}_k :=
\hat{\tau}_k(Q_n) =
\big(q_{[n]}\big)^{\abs{\tilde{\chi}({Q_{n-1}^{(k-2)}})}}\big(y_{[n]}z_{[n]}\big)^{\abs{\tilde{\chi}({Q_{n-1}^{(k-1)}})}} \prod_{\substack{S\subseteq [n] \\ \abs{S}>k}}{u_S}^{{\binom{\abs{S}-2}{k-1}}}
\end{equation}
where \(u_S=\sum_{i \in S} q_i\left(y_i^{-1}+z_i^{-1}\right).\) Moreover, $\hat{\tau}_0:=\hat{\tau}_0(Q_n)=\prod_{i \in [n]}{(y_i+z_i)}$.}
\medskip

\begin{proof}
First, we reformulate the right-hand side of \eqref{geneftncube}.  Define $\hat{D}_S=\sum_{k=0}^{\abs{S}}{\binom{\abs{S}}{k}x^{k+1}}=(1+x)^{\abs{S}}x$ for $S \subseteq [n]$ with $S \ne \emptyset$. Note that for $i \in [n]$, we have $\hat{D}_{\{i\}}=(1+x)x$ and $\log(u_i)=\log(y_i+z_i)+\log(q_i)-\log(y_iz_i)$. By \cite[Theorem~4.2]{DKM2}, for the weighted total Laplacians $\hat{L}^{\rm tot}_k$, we have

\begin{itemize}
\item $\det \hat{L}^{\rm tot}_{-1}=\prod_{i \in [n]}{(y_i+z_i)}=\hat{\tau}_{0}$,
\item $\det \hat{L}^{\rm tot}_0=\big( \prod_{i \in [n]}{(y_i+z_i)}\big)\big(\prod_{S\subseteq [n] }u_S\big)$, and
\item $\det \hat{L}^{\rm tot}_k=\big(\prod_{S\subseteq [n]}u_S^{\binom{\abs{S}}{k}}\big) \mbox{ for } k \in [1,n]$,
\end{itemize}
where we define $u_{\emptyset}:=1$. Therefore,
\begin{align*}
\hat{D}(x) 
&=(1+x)\sum_{i \in [n]}\log{(y_i+z_i)}+ \sum_{\substack{S \subseteq [n]\\\abs{S} =1}}{(\log u_S)\hat{D}_S} + \sum_{\substack{S \subseteq [n]\\\abs{S} \geq 2}}{(\log u_S)\hat{D}_S}\\
&=(1+x)^2\sum_{i \in [n]}\log{(y_i+z_i)}+\sum_{i \in [n]}{(1+x)x\big(\log(q_i)-\log(y_iz_i)\big)}+\sum_{\substack{S \subseteq [n]\\\abs{S} \geq 2}}{(\log u_S)\hat{D}_S}.
\end{align*}
Applying Proposition~\ref{genecube}, we obtain
\begin{equation}\label{re-K(x)}
\hat{K}(x)=\hat{\kappa}_0+\big(\log(q_{[n]})\big)K_q+\big(\log(y_{[n]}z_{[n]})\big)K_{yz}+ \sum_{\substack{S\subseteq [n]\\ \abs{S} \geq 2}}{(\log u_S)K_S}
\end{equation}
where $K_q=(1+x)^{-1}(x+F_{\epsilon}(x))$, $K_{yz}=(1+x)^{-1}(-x+F_{\nu}(x))$, and $K_S=(1+x)^{-2}\hat{D}_S$ for each $S \subseteq [n]$ with $\abs{S} \geq 2$. Thus, the coefficients of $K_q$, $K_{yz}$, and $K_S$ are the exponents of $q_{[n]}$, $y_{[n]}z_{[n]}$, and $u_S$, respectively, and the proof reduces to showing that these exponents are as claimed in \eqref{formula-cube}.

Let us compute the exponent of $q_{[n]}$ in \eqref{formula-cube}. For $k \in [0,n-1]$, let $c_k$ be the number of $k$-cells in $Q_{n-1}$ and $c_{-1}=1$.  Thus $c_{k-1}=\epsilon_{k}$ for all $k \in [n]$, and $c_{-1}=1+\epsilon_0$.  Observe that
\begin{align*}
K_q &= (1+x)^{-1}(x+F_{\epsilon}(x)) \\
&= (1+x)^{-1}\sum_{k=-1}^{n-1}{c_{k}x^{k+2}} \\
&= \sum_{k=1}^{n}\left(\sum_{\ell=-1}^{k-2}{(-1)^{k-2-\ell}c_{\ell}}\right)x^k=\sum_{k=1}^{n}{\abs{\tilde{\chi}({Q_{n-1}^{(k-2)}})}x^{k}}
\end{align*} 
where the last equality is given by~\eqref{euler-char-sign}.  Hence $\abs{\tilde{\chi}({Q_{n-1}^{(k-2)}})}$ is the exponent of the product $q_{[n]}$ in \eqref{formula-cube}. 

Next, we consider the exponent of $y_{[n]}z_{[n]}$ in \eqref{formula-cube}. Let $c_0'=c_0-1$, and $c'_{k}=c_{k}$ for all $k \in [n]$. By a similar calculation for $K_q$, we have
\begin{align*}
K_{yz} &= (1+x)^{-1}(-x+F_{\nu}(x)) \\
&= (1+x)^{-1}\left(-x+\sum_{k=0}^{n}c_{k}x^{k+1}\right) \\
&= (1+x)^{-1}\left(\sum_{k=0}^{n}c'_{k}x^{k+1}\right) \\
&= \sum_{k=1}^{n}\left(\sum_{\ell=0}^{k-1}{(-1)^{k-1-\ell}c'_{\ell}}\right)x^{k} \\
&= \sum_{k=1}^{n}\left(\sum_{\ell=-1}^{k-1}{(-1)^{k-1-\ell}c_{\ell}}\right)x^{k}=\sum_{k=1}^{n}{\abs{\tilde{\chi}({Q_{n-1}^{(k-1)}})}x^{k}}
\end{align*}
where again the last equality uses~\eqref{euler-char-sign}.  Hence $\abs{\tilde{\chi}({Q_{n-1}^{(k-1)}})}$ is the exponent of the product $y_{[n]}z_{[n]}$ in  \eqref{formula-cube}.

Finally, let $S \subseteq [n]$ with $\abs{S} \geq 2$. Recall that $\hat{D}_S=(1+x)^{\abs{S}}x$. For $k \in [n]$, the coefficient of $x^k$ in $K_S=(1+x)^{-2}\hat{D}_S=(1+x)^{\abs{S}-2}x$ is $\binom{\abs{S}-2}{k-1}$, which is the exponent of $u_S$ in the last product of \eqref{formula-cube}.
\end{proof}

\begin{remark}
Theorem~\ref{inter-hyper-theorem} provides an alternate proof of Theorem~\ref{thm:cube-tree}. Simply applying Proposition~\ref{thm-Euler} gives 
\begin{equation}\label{formula-cube2}
\hat{\tau}_k=\big(q_{[n]}\big)^{\sum_{i=k-1}^{n-1}\binom{n-1}{i}\binom{i-1}{k-2}} \big(y_{[n]}z_{[n]}\big)^{\sum_{i=k}^{n-1}{\binom{n-1}{i}\binom{i-1}{k-1}}} \prod_{\substack{S\subseteq [n] \\ \abs{S}>k}}{u_S}^{\binom{\abs{S}-2}{k-1}}.
\end{equation}
This equation, together with a binomial identity, yields Theorem~\ref{thm:cube-tree}.
\end{remark}

\begin{remark}
We can obtain \eqref{formula-cube2} directly from \eqref{re-K(x)}. First note that
\[
F_{\epsilon}(x)=\sum_{k=1}^{n}\binom{n-1}{k-1}2^{n-k}x^{k+1}=x^2(x+2)^{n-1}=
\sum_{i=0}^{n-1}\binom{n-1}{i}(1+x)^{i}x^2.
\]
(This identity can be obtained either combinatorially, by counting the faces of $Q_{n-1}$ in two different ways, or algebraically, by equating two different binomial expansions of $x^2(x+2)^{n-1}$.)  It follows that $K_q=x+\sum_{i=1}^{n-1}\binom{n-1}{i}(1+x)^{i-1}x^2$. Hence, for $k \in [n]$, the coefficient of $x^{k}$ in $K_q$ is $\sum_{i=k-1}^{n-1}\binom{n-1}{i}\binom{i-1}{k-2}$, which is the exponent of $q_{[n]}$ in~\eqref{formula-cube2}.  Similarly, since $F_{\nu}(x)=\sum_{i=0}^{n-1}\binom{n-1}{i}(1+x)^{i}x$, we obtain $K_{yz}=\sum_{i=1}^{n-1}\binom{n-1}{i}(1+x)^{i-1}x$.  For $k \in [n]$, the coefficient of $x^{k}$ in $K_{yz}$ is $\sum_{i=k}^{n-1}\binom{n-1}{i}\binom{i-1}{k-1}$, 
which is the exponent of $y_{[n]}z_{[n]}$ in~\eqref{formula-cube2}. 
\end{remark}

\bibliographystyle{alpha}
\bibliography{biblio}
\end{document}